\documentclass[a4paper,11pt,reqno]{amsart}
\usepackage{mathpazo}
\usepackage{textcomp}
\usepackage{nicefrac} 
\usepackage{xfrac}  
\usepackage{frcursive}
\usepackage[T1]{fontenc}

\usepackage{graphicx}
\usepackage{subcaption}
\usepackage{mwe}

\usepackage[hidelinks]{hyperref}
\usepackage{amsmath}
\usepackage{amsfonts}
\usepackage{amssymb}
\usepackage{makerobust}
\usepackage{enumerate}
\usepackage{mathrsfs}

\usepackage[all]{xy}

\usepackage{comment}

\MakeRobustCommand\overrightarrow

\usepackage{mathtools}

\usepackage[textsize=tiny]{todonotes}

\theoremstyle{plain}
\newtheorem{lemma}{Lemma}[section]
\newtheorem{theorem}[lemma]{Theorem}
\newtheorem{proposition}[lemma]{Proposition}

\newtheorem*{convention*}{Convention}

\theoremstyle{definition}
\newtheorem{definition}[lemma]{Definition}

\newtheorem{example}[lemma]{Example}
\newtheorem{remark}[lemma]{Remark}
\newtheorem*{definition*}{Definition}

\theoremstyle{remark}

\newcommand{\smallO}[1]{\ensuremath{\mathop{}\mathopen{}o\mathopen{}\left(#1\right)}}

\newcommand{\toto}{\rightrightarrows}

\newcommand{\func}[1]{\mathcal{C}^{\tiny \infty}(#1)}

\usepackage{enumitem}
\usepackage{enumerate}

\title[Poisson Integrators in Mechanics]{Numerical Methods in Poisson Geometry and their Application to Mechanics}

\author[O. Cosserat]{Oscar~Cosserat} 
\address{O. Cosserat:  LaSIE  -- CNRS \& La Rochelle University,
Av. Michel Cr\'epeau, 17042 La Rochelle Cedex 1, France}
\email{oscar.cosserat@univ-lr.fr}

\author[C. Laurent-Gengoux]{Camille~Laurent-Gengoux}
\address{C. Laurent-Gengoux:  Institut Elie Cartan de Lorraine (IECL), UMR 7502 --  3 rue Augustin Fresnel, 57000 Technop\^ole Metz, France}
\email{camille.laurent-gengoux@univ-lorraine.fr}

\author[V. Salnikov]{Vladimir~Salnikov} 
\address{V. Salnikov:  LaSIE  -- CNRS \&  La Rochelle University,
Av. Michel Cr\'epeau, 17042 La Rochelle Cedex 1, France}
\email{vladimir.salnikov@univ-lr.fr}

\date{\today}

\begin{document}

\begin{abstract}
We recall the question of geometric integrators in the context of Poisson geometry, and explain their construction. These Poisson integrators are tested in some mechanical examples. Their properties are illustrated numerically and they are compared to traditional methods.\\
\end{abstract}

\maketitle

\tableofcontents

\newpage


\section*{Introduction}

 Poisson geometry allows to describe a large class of conservative systems in mechanics, both for discrete and continuous media. Those may be obtained as a result of a reduction procedure or as an ad hoc model for evolution of some natural systems. Just to cite a few non-standard examples: chemistry (polymer dynamics, \cite{Ratiu2011}), plasticity (elastoplasticity, \cite{Liu2013}), population dynamics (\cite{vanhaecke2016}), liquid crystals theory (\cite{Ratiu2013}), thermodynamics (GENERIC formalism, \cite{Grmela2018}), control theory (active and kinematic constraints, \cite{Marle1990}). Moreover, the geometry of the Poisson structure matters to express symmetries, conservation laws and qualitative behaviour of dynamical systems.  
 
 On top of the purely mathematical significance of Poisson structures, for Hamiltonian differential equations, they provide the appropriate geometry to be preserved in numerical simulation thus potentially resulting in a very broad class of geometric integrators. The state of the art in this context concerns mostly \emph{symplectic integrators}, known for decades now (\cite{Verlet1967}), they correspond to symplectic structures, i.e. non-degenerate Poisson for systems defined on a canonical phase space, geometrically meaning on a cotangent bundle. To the best of our knowledge there are very few works treating more general Poisson structures. Most of them rely on Weinstein's splitting theorem (see for example \cite{daSilva1999}), which states that a Poisson manifold can be foliated by symplectic leaves, the natural idea that emerges is: ``restrict to the leaf containing the trajectory, apply symplectic integration on it''. 
 
 There are two conceptual difficulties with this ``naive'' approach to construct \emph{Poisson integrators}: the mentioned foliation is very rearly explicitly known, and in the generic situation it is singular, meaning that the successfull applications of the above strategy are rather exceptional particular cases. A way out was proposed in \cite{oscar}, it is based on integration of Poisson manifolds to (local) symplectic groupoids, geometrically this procedure can be viewed as a kind of desingularisation. 
 
 The goal of the current paper is to explain the subtleties occurring in this integration procedure while making it explicit and also present the features of resulting Poisson integrators. We will not go into mathematical proofs (a motivated reader is invited to consult \cite{oscar} for details), but rather focus on the precise behaviour of the constructed numerical methods. We will however provide all the necessary building blocks to make the paper self-consistent. 

 We can address one important detail already here: the mere definition of what is an appropriate notion of structure preserving numerical method in the context of Poisson geometry. For the symplectic case the usual strategy is to say that the the preservation of the symplectic form guarantees the energy conservation. In fact it is a bit subtler than that: a discrete symplectic flow indeed preserves the level surface of some Hamiltonian (suppose there is no topological obstruction for its global existence), but not necessarily the same that corresponds to the energy of the system and was defining the evolution of it. The difference however can be estimated: it will be oscillating around zero with an amplitude that can be made small for exponentially long time. For the Poisson case this subtlety is even more pronounced: the preservation of the Poisson structure alone does not guarantee the energy conservation, roughly speaking this is due to the existence of the ``degenerate'' directions. 
 
 We thus introduce a stronger notion of a \emph{Poisson Hamiltonian integrator} (PHI) -- a numerical method for which discretising a trajectory of a Hamiltonian system preserves the Poisson structure and coincides exactly with the trajectory of some time-dependent Hamiltonian, which equals to the initial one at the order of the method. This latter condition seems to be very strong, but in fact it is not: the Poisson Hamiltonian integrator requires the existence of such a time-dependent Hamiltonian but not constructively, in fact, only a theoretical existence is enough: we do not need, in general, to compute it. This unspecified existence guarantees (see again \cite{oscar}) that the discretisation preserves the geometric invariants of the phase space (symplectic leaves, leafwise symplectic structures) together with the qualitative physical properties of the system -- we will illustrate this in different situations. 

 The paper is organized as follows: 
We start with an illustration of our PHI technique for a simple Lotka-Volterra system -- we observe a better behaviour of a numerical solution in comparison with the standard Runge-Kutta method.
In section \ref{sec:poisson} we recall the minimal working knowledge from Poisson geometry to formulate precisely the definition and give the strategy of construction of the Poisson Hamiltonian integrators. Then in section \ref{sec:PHI} we introduce more advanced geometric notions that are needed for the construction which is made explicit in section \ref{sec:PHI-constr}. The last section \ref{sec:num} is devoted to numerical results of comparison of the constructed Poisson Hamiltonian integrators with standard methods for various typical situations.

\section{An introductory example} \label{sec:LVnum}

Let us look at a particular case of Lotka-Volterra type equations
\begin{equation}\label{eq:Lodka3} \nonumber
    \begin{array}{lll}
        \dot{x}_1 &= x_1(x_2 + x_3)\\
        \dot{x}_2 &= x_2(-x_1 + x_3)\\
        \dot{x}_3 &= -x_3(x_1 + x_2). 
    \end{array}
\end{equation}
This system of differential equations appears in \cite{Volterra1931} (page 97, equation (16)) as a model in population dynamics, and similar systems have been extensively studied since then. 
For this particular one, an explicit solution was computed in \cite{vanhaecke2016}: it allows to compare, for any initial point and desired time, numerical simulations with the exact solution  ${x}^{\hbox{\tiny{exact}}}(t)= (x_1(t),x_2(t),x_3(t))$. \\
Some integral curves  $ {x}^{\hbox{\tiny{exact}}}(t)$ go to infinity, exploding exponentially fast while approaching some specific time. For instance, using the exact formulas given in \cite{vanhaecke2016}, one observes that for the initial value
$$
\begin{pmatrix}
    x_1(0)\\
    x_2(0)\\
    x_3(0)
\end{pmatrix}
= 
\begin{pmatrix}
    -3\\
    5\\
    10^{-3}
\end{pmatrix},
$$
the trajectory $  {x}^{\hbox{\tiny{exact}}} (t)$ starts exploding around $T_{sing} \simeq 0.23$. We use this singularity to test two numerical methods of different nature.

The first one is the standard explicit 2nd order Runge--Kutta method (RK-2); the second one is a numerical scheme (PHI-1) where ${{x}}_{n+1}^{\hbox{\tiny{PHI}}}$ is implicitely computed from ${{x}}_{n}^{\hbox{\tiny{PHI}}}$ by the following two step procedure: 

$$
 \left\{\begin{array}{lll}
        x_{1,n+1} &= e^{(\frac{\Delta t}{2}({y_{2,n}} + {y_{3,n}}))} {y_{1,n}}  \\
        x_{2,n+1} &= e^{(\frac{\Delta t}{2}(-{y_{1,n}} + {y_{3,n}}))} {y_{2,n}}\\
        x_{3,n+1} &= e^{(\frac{-\Delta t}{2}({y_{1,n}} + {y_{2,n}}))} {y_{3,n}} 
    \end{array}
    \right.
    $$
    with the intermediate point $\mathbf{y}_n = (y_{1,n},y_{2,n}, y_{3,n})$ subject to the relation: 
$$
    \left\{\begin{array}{lll}
        e^{(\frac{-\Delta t}{2}(y_{2,n} + {y_{3,n}}))} {y_{1,n}} &= x_{1,n}\\
        e^{(\frac{\Delta t}{2}({y_{1,n}} - {y_{3,n}}))} {y_{2,n}} &= x_{2,n}\\
        e^{(\frac{\Delta t}{2}({y_{1,n}} + {y_{2,n}}))} {y_{3,n}} &= x_{3,n} 
    \end{array}
    \right.
$$

This numerical scheme is of order 1 only.
Nevertheless, it performs much better than the order two Runge-Kutta method near the singularity $T_{sing} $ described above (see Fig. \ref{fig:Lodka_dist_to_real_gprd1}):

\begin{figure}[h]
    \centering
    \includegraphics[width=0.5\textwidth]{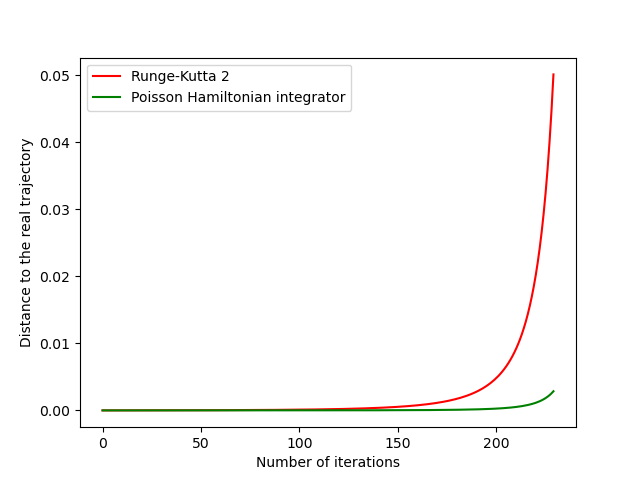}
    \caption{Comparison of the error for RK-2 and PHI-1}
    \label{fig:Lodka_dist_to_real_gprd1}
\end{figure}
 We observe that the PHI-1 method approximates the solution much better than the RK-2. In fact comparing the values of the variables, we see that the RK-2 misses the singularity completely, in a sense that it goes off the exact solution much earlier than it may tend to infinity, so using it alone one would not even notice that the solution is singular; while the PHI-1 method pushes the solution up to the last step before hitting the singularity, where it goes to what one can call ``numerical infinity''. 

It is actually quite unexpected that an order $1$ integrator behaves better than an order $2$ integrator. The same situation happens when one compares PHI-1 one with a 4th order Runge-Kutta scheme for bigger values of $T.$ The explanation is that as often in such situations when such a behavior happens, we are comparing an integrator that does not preserve the underlying geometric structure of the differential equations with an integrator that does. We will see below that the second method is a \emph{Poisson Hamiltonian integrator}, so we will understand what is preserved exactly and why this guarantees an appropriate trajectory.
The aim of the following sections is to explain what are such integrators, which properties they satisfy, and how one can construct and implement them in a wide class of examples.

\section{Geometric integrators for Poisson Hamiltonian systems.
}
\label{sec:poisson}

\subsection{Poisson structures: definition}

Let us briefly recall\footnote{A reader familiar with Poisson geometry may safely skip this section up to \ref{sec:PI}.} what are Poisson bivector fields and Hamiltonian differential equations, and why they matter.  
In mechanics,  quite some differential equations governing a motion $x(t)$ valued in an open subset $\mathcal U\subset \mathbb R^n $  take the form:

\begin{equation}\label{eq:HamiltonianDiffEquation} 
\left\{\begin{array}{rcl} 
\dot{x}_1(t)& =& \sum\limits_{j=1}^n \pi_{1j}(x(t)) \, \frac{ \displaystyle \partial H (x(t)) }{ \displaystyle \partial x_j}   \\ &\vdots & \\ \dot{x}_n (t)& =&  \sum\limits_{j=1}^n \pi_{nj}(x(t)) \, \frac{\displaystyle \partial H (x(t)) }{\displaystyle  \partial x_j}  
\end{array} \right. 
\end{equation}
where $H \colon \mathcal U\longrightarrow \mathbb R$ is a smooth function, that it is customary to call {\emph{Hamiltonian}} in this context,  and $(\pi_{ij}), \; i,j \in \{1, \dots, n\}$ is a family of smooth functions which satisfy
 \begin{equation}
 \label{eq:conditions}
 \pi_{ji}= - \pi_{ij} \; \text{  and  } \; \sum_{a=1}^n \frac{\partial \pi_{ij}}{\partial x_a}  \,  \pi_{ak} + \circlearrowleft = 0,
 \end{equation} 
 where all indices $i,j,k \in \{1, \dots, n\}$ and $\circlearrowleft$ stands for their cyclic permutations.. 
 
 Let us first explain these conditions: there is a theorem (see chapter 1 of \cite{LPV2012}) claiming that a family $(\pi_{ij}), i,j \in \{1, \dots, n\}$ of smooth functions on $\mathcal U\subset \mathbb R^n$ satisfy (\ref{eq:conditions}) if and only if the bilinear map:
  $$  
  \begin{array}{rcll} \{\cdot, \cdot\} \colon& \func{\mathcal{U}} \times \func{\mathcal{U}} &  \longrightarrow & \func{\mathcal{U}} \\ & f,g & \mapsto & \sum\limits_{i,j=1}^n \frac{\displaystyle 1}{\displaystyle 2} \pi_{ij}(x)  \left( \frac{\displaystyle\partial f}{\displaystyle\partial x_i} \frac{\displaystyle\partial g}{\displaystyle\partial x_j} - \frac{\displaystyle\partial g}{\displaystyle\partial x_i} \frac{\displaystyle\partial f}{\displaystyle\partial x_j}\right) \\ \end{array}
  $$
  is a Lie bracket, i.e. is anti-symmetric and satisfies the Jacobi identity:
  \begin{equation}
  \label{eq:Jacobi}
       \{f,g\} = -\{g,f\} \hbox{ and } \{\{f,g\},h\} + \{\{g,h\},f\}+ \{\{h,f\},g\} = 0
  \end{equation}
  for all $ f,g,h \in \func{M}$.
  Equivalently, the functions $ \pi_{ij}(x)$ above can be considered as being a \emph{tensor} (i.e. a map from $\mathcal U \subset \mathbb R^n$ to ${\mathfrak{gl}}_n(\mathbb R)$, 
  viewed as  $n \times n$  matrices depending on $x$:
  $$\begin{array}{rcll}\pi \colon &\mathcal U& \longrightarrow& {\mathfrak{gl}}_n(\mathbb R)  \\ & x&  \mapsto& (\pi_{ij}(x)_{i,j=1}^n) \end{array}$$
  
\begin{definition}
    A \emph{Poisson structure} on an open subset $ \mathcal U \subset \mathbb R^n$ is  a  tensor $\pi
    $ on $\mathcal U$ such that the bilinear map  $\func{\mathcal U} \times \func{\mathcal U} \to \func{\mathcal U}$ defined by $$\{f,g\}(x) = {}^t \nabla_x f . \pi(x) .\nabla_x g $$ satisfies the skew-symmetry and Jacobi identities (\ref{eq:Jacobi}).
    Equivalently, Poisson structures can be seen as tensors whose coefficients $(\pi_{ij}(x)), i,j \in \{1, \dots, n\} $ satisfy Equations (\ref{eq:conditions}).
    \end{definition}
The long history behind this notion comes with a vocabulary, which is sometimes confusing:
it is customary to call the bilinear map $\{\cdot, \cdot\} $ the \emph{Poisson bracket}. Also, functions in $\func{\mathcal U} $ are -- depending on the context -- sometimes all called \emph{Hamiltonian functions} or simply \emph{Hamiltonians}.
Last, the easy to check relation:
 $$ \{f,gh\}= \{f,g\} h + g \{f,h\} $$
 is called \emph{Leibniz identity}. 
For examples of Poisson structures that appear in mechanics, see Section \ref{sec:poisson-examples}.

We saw that a Poisson structure associates to \emph{two} Hamiltonian function  $f,g \in \func{\mathcal{U}}$ another Hamiltonian function $\{f,g\} $. But it also allows to associate to \emph{one} Hamiltonian function $H \in \func{\mathcal{U}}$ a first order autonomous differential equation, as in \eqref{eq:HamiltonianDiffEquation}. 
More abstractly, \eqref{eq:HamiltonianDiffEquation} means that to a Hamiltonian function $H$ we associate the (vector) differential equation: 
 \begin{equation} \label{eq:ham-inv}
      \dot{x}(t) = \pi(x(t)) \cdot \nabla_{x(t)} H. 
 \end{equation}
 
 We say that a first order autonomous differential equation of the form \eqref{eq:ham-inv} above is a  \emph{Hamiltonian differential equation for $(\pi,H) $}.

\begin{remark}In differential geometry, a first order differential equation on an open subset $U$ of $\mathbb{R}^n$ with an independent parameter having the meaning of time
\begin{equation}\label{eq:diff_eq}
    \dot{x} = F(x)
\end{equation} 
is generally referred to as a vector field on $U$. 
Moreover, rather than considering only open subsets of $\mathbb{R}^n$, the phase space is often assumed to be a differential manifold. 
\end{remark}

\subsection{The underlying geometry of a Poisson structure}

It is natural to ask why it matters that behind an autonomous first order differential equation, there is a Poisson structure and a Hamiltonian function. What does one gain by knowing   that a given differential equation is Hamiltonian of $(\pi, H)$? The classical answer is that many ``quantities'' related to the $\pi$ or $H$ are preserved under the flow of the differential equation. 

More precisely, any solution $x(t) $ (often called ``integral curve'' in mathematics) of a differential equation Hamiltonian for $(\pi,H) $ preserves several functions:
\begin{enumerate}
      \item $H$ is a constant of motion, i.e. $H(x(t)) = H(x(0))$. In words, it means that the flow of a Hamiltonian differential equation for $(\pi, H) $ preserves the level sets of the Hamiltonian function $H$.
      \item More generally, any function $G$ such that $\{G,H\}=0 $ is a constant of motion.
      \item Even more generally, for any function $G \in \func{\mathcal{U}}$, one hav
      $$
        \frac{\mathrm{d} G(x(t))}{\mathrm{d}t} = \{G,H\} (x(t)).
        $$
\end{enumerate}

Now, let us recall some properties of the time $t$ flow of a differential equation which is Hamiltonian for $(\pi,H) $, i.e. the map $ \phi_t \colon x(0) \mapsto x(t)$, which is well-defined in a neighborhood of any $m \in \mathcal U$ for $t$ small enough:  
\begin{enumerate}
      \item $\phi_t $ preserves $\pi$: $\pi(\phi_t(x)) = {}^t \nabla_x \phi_t (x). \pi(x) . \nabla_x \phi_t (x)$. In words, it means that the flow of a Hamiltonian differential equation for $(\pi,H) $ preserves the Poisson structure $\pi$.
      \item The previous condition can also be stated as meaning that the pull-back map $f \mapsto \phi_t^* f $, i.e. the map assigning to a smooth Hamiltonian function  $f$ the Hamiltonian function $f \circ \phi_t $, is a Lie algebra morphism, i.e. $$ \phi_t^* \{f,g\} =  \{\phi_t^* f,\phi_t^* g\}$$
      for all functions $ f,g$.
      \item Item 1 above means in particular that the geometry of $\pi$ is preserved. For instance if at the initial point $x(0)$, the matrix $\pi(x(0))$ has some given rank, it has this same rank at every point along the integral curve $x(t)$. Below, we will give a much stronger statement, using the notion of symplectic leaves.
\end{enumerate}

A \emph{symplectic singular foliation} on $\mathcal U\subset \mathbb R^n $ is a partition
 $\mathcal U= \sqcup_{c \in I} \mathcal{F}_c$  by submanifolds, such that each $\mathcal F_c $ is equipped with a symplectic structure $ \omega_{c}) $.
 The pair $(\mathcal F_c,\omega_c) $ is called a \emph{symplectic leaf}.

\begin{theorem}[\cite{LPV2012, daSilva1999}]
Any Poisson structure on $\mathcal U\subset \mathbb R^n $ induces a natural foliation by symplectic leaves characterized by the following two properties:
\begin{enumerate}
   \item two points belong to the same symplectic leaf if and only if they can be connected by a sequence of Hamiltonian trajectories, i.e. by integrating the equation \refeq{eq:ham-inv} for some choice of Hamiltonian functions. 
   \item for every $c \in I$, the inclusion $i \colon \mathcal{F}_c \xhookrightarrow{} M$ is a Poisson morphism.
\end{enumerate}
In addition, the tangent space of the symplectic leaf $\mathcal F_c $ at a point $m$ coincides with the image of the matrix $\pi(m) $.
\end{theorem}

\begin{figure}[h]
    \centering
    \includegraphics[width=0.5\textwidth]{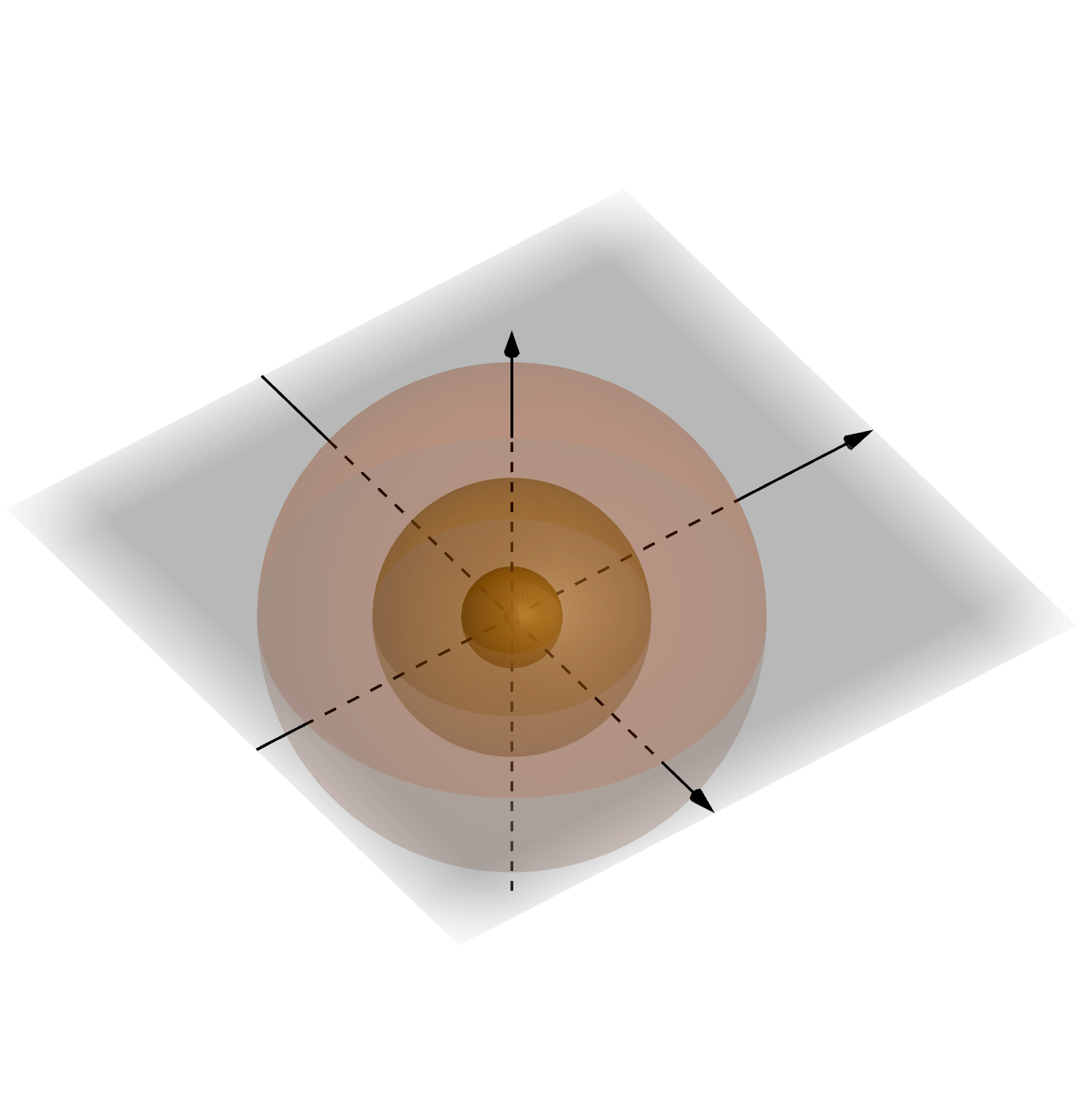}
    \caption{A foliation of $\mathbb{R}^3$ by concentric spheres and the origin.}
    \label{fig:foliation}
\end{figure}

As observed in some of the following examples, the foliation is generically singular. Two neighbouring leaves do not necessarily have the same dimension and can differ from a topological point of view. Therefore, its study is an active field of research and motivates one of the long term applications of the numerical tools we present here.

The last reason explaining the importance of knowing that a differential equation is  Hamiltonian for $(\pi,H)$ is the following: a solution $x(t)$ of a differential equation Hamiltonian for $(\pi,H)$ can not ``jump'' from one symplectic leaf to another. That is if $x(0)$ belongs to a leaf $\mathcal F_c $, then the solution $x(t)$ belongs to the same leaf for every $t$. 

\begin{remark}
This last ``constraint'' is maybe less studied for numerical methods than the previous ones, because when the Poisson is symplectic, it is not a constraint at all: the foliation contains only one leaf being the whole space.  But for non-symplectic Poisson structures this is a very important feature to take under account.
\end{remark}

In conclusion, for any Poisson structure $\pi$: 
\begin{enumerate}
    \item[$\diamondsuit$]  for any first order autonomous differential equation which is Hamiltonian for $(\pi,H)$, the Hamiltonian $H$ is constant along any integral curve;
    \item[$\heartsuit$]  each integral curve stays on the same symplectic leaf of the foliation defined by $\pi$;
    \item[$\spadesuit$] the flow of this differential equation preserves $\pi$, i.e. is a Poisson morphism; 
    \item[$\clubsuit$] the converse is not necessarily true: preservation of $\pi$ does not guarantee that the flow is Hamiltonian.
\end{enumerate}

\subsection{Examples of Hamiltonian equations -- first candidates for integrators. } \label{sec:PI}

The goal of what follows is to explain the logic behind the construction of numerical schemes that take into account the geometric features described above. We illustrate it on simple cases yet instructive examples. 

Important examples of Poisson structures are the symplectic ones in their canonical form, e.g. where $\pi = \begin{pmatrix}
0 & -I\\
I & 0
\end{pmatrix}$. For those, a wide example of symplectic integrators are already available in the literature. One construction of such integrators uses the principle of symplectic Runge--Kutta schemes (\cite{yoshida}): 
\begin{equation}
    \begin{array}{ll}
         x_{n+1} &= x_n + \Delta t \sum\limits_{i=1}^s b_i k_i\\
         k_i &= \pi(x_n + \Delta t \sum\limits_{j=1}^n a_{ij} k_j). \nabla H(x_n + \Delta t \sum\limits_{j=1}^n a_{ij} k_j)
    \end{array}
\end{equation}
where slopes $k_i$ are implicitly defined and coefficients $b_i$ and $a_{ij}$ are chosen such that the discrete flow preserves $\pi$. For this precise $\pi$, any trajectory preserving it is necessary a time-dependent Hamiltonian one, at least locally.

When $\pi$ is degenerate, the same principle can be applied (\cite{Jay2003}) and leads to a discrete flow that preserves the Poisson tensor. However, it would lead to non-physical simulations, e.g. non-Hamiltonian ones. It does not guarantee the Hamiltonian property of the discrete trajectory anymore because of the existence of outer Poisson automorphisms, as illustrated in the following example.

\begin{example} \label{poisnoham}
    Consider $\mathcal U= \mathbb{R}^3$. The system of differential equations
    \begin{equation}\label{eq:x2y2_modified} 
    \left\{\begin{array}{rcl} 
    \dot{x}&= & -\frac{(x+y-z)}{8}\left( (-x-y+z)^2 + (x + y - z)^2\right) \\ 
    \dot{y}& =& \frac{(-y+z)}{4}\left((x-y+z)^2 + (x + y - z)^2\right) \\
    \dot{z}& =& \frac{(x - y + z)}{8}\left((x - y + z)^2 + (x + y -z)^2\right)
         \end{array}\right.
    \end{equation}
    is Hamiltonian with respect to the Poisson structure
    $$\pi(x,y,z) = \frac{(x - y + z)^2 + (x + y -z)^2}{4} 
    \begin{pmatrix}
0 & -1  & -1 \\
1 & 0 & -1 \\
1 & 1 & 0
\end{pmatrix}$$ and the Hamiltonian
$H \colon (x,y,z) \mapsto \frac{(x - y + z)^2 + (x + y - z)^2}{8}.$ 
For any $\Delta t >0 $, the system of equations:

\begin{equation}\label{eq:Poisson_integrator_for_quad}
\left \{
    \begin{array}{lllll}
        x_{n+1} =& x_n \cos{\left(\Delta t \frac{x_n^2 + y_n^2 - 2 y_n z_n + z_n^2}{2}\right)} + y_n \sin{\left(\Delta t \frac{x_n^2 + y_n^2 - 2 y_n z_n + z_n^2}{2}\right)} \\
        &- z_n \sin{\left(\Delta t \frac{x_n^2 + y_n^2 - 2 y_n z_n + z_n^2}{2}\right)}  \\
        y_{n+1} =& \frac{-x_n + y_n - z_n}{2} \sin{\left(\Delta t \frac{(x_n - y_n + z_n)^2 + (x_n + y_n - z_n)^2}{4}\right)} + \frac{-x_n + y_n + z_n}{2} \exp{(\Delta{t}^k)} \\
        &+ \frac{x_n + y_n - z_n}{2}\cos{\left(\Delta{t} \frac{(x_n - y_n + z_n)^2 + (x_n + y_n - z_n)^2}{4}\right)}\\
        z_{n+1} =& \frac{-x_n + y_n + z_n}{2}\exp{(\Delta{t}^k)} + \frac{x_n - y_n + z_n}{2}\cos{\left(\Delta{t} \frac{(x_n - y_n + z_n)^2 + (x_n + y_n - z_n)^2}{4}\right)} \\
        &+ \frac{x_n + y_n - z_n}{2} \sin{\left(\Delta{t} \frac{(x_n - y_n + z_n)^2 + (x_n + y_n - z_n)^2}{4}\right)}
    \end{array}
    \right.
\end{equation}
is a discretisation of order $k$ of the differential equation \eqref{eq:x2y2_modified}.
It is routine  to check that it is a Poisson integrator, i.e. $(x_n,y_n) \mapsto  (x_{n+1},y_{n+1} )$ is a Poisson isomorphism. However, it is \emph{not} a Poisson Hamiltonian integrator.
This can be proven as follows: for any vector field on $\mathbb R^3$ that vanishes at least quadratically at $ (0,0, 0)$, the differential of its flow at $(0,0,0)$ is the identity map. In particular,  since the coefficients of the Poisson structure vanish at least quadratically at $ (0,0,0)$, so does any Hamiltonian vector field, so that any Poisson Hamiltonian integrator should be made of a local diffeomorphism whose differential at $ (0,0,0)$ is the identity map. Since the differential at $(0,0,0) $ of the map $(x_n,y_n, z_n ) \mapsto  (x_{n+1},y_{n+1}, z_{n+1} ) $ is not equal to identity map, the latter Poisson integrator is not Hamiltonian.

\begin{figure}[htp]
        \centering
        \begin{subfigure}[b]{0.475\textwidth}
            \centering
            \includegraphics[width=\textwidth]{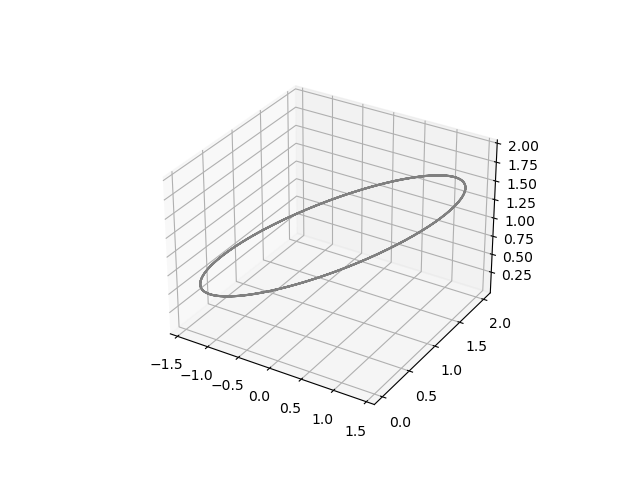}
            \caption{Flow of \eqref{eq:x2y2_modified}}
        \end{subfigure}
        \hfill
        \begin{subfigure}[b]{0.475\textwidth}  
            \centering 
            \includegraphics[width=\textwidth]{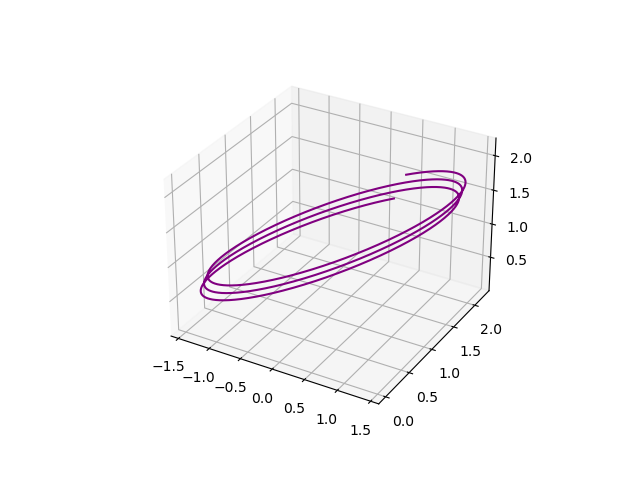}
            \caption{Poisson integrator \eqref{eq:Poisson_integrator_for_quad}}
        \end{subfigure}
        \caption{Discrepancy created by the Poisson integrator \eqref{eq:Poisson_integrator_for_quad}
        }
        \label{fig:comparison_flow_integrator}
\end{figure}

Figures \ref{fig:comparison_flow_integrator}.A \ref{fig:comparison_flow_integrator}.B show the difference between the actual flow of \eqref{eq:x2y2_modified} and the first iterations of the Poisson integrator \eqref{eq:Poisson_integrator_for_quad} at order $k = 2$ with initial points $\begin{pmatrix}
    1\\
    1\\
    2
\end{pmatrix}$ and timestep $\Delta t = 10^{-4}.$ The flow should be $4\pi$-periodic while an approximation of it at order 2 destroys the topology of the curve, even while it preserves the Poisson tensor. The geometric reason is that the Poisson integrator
\eqref{eq:Poisson_integrator_for_quad}
does not stay on a symplectic leaf of the Poisson structure, i.e. a hyperplane of equation $\{x-y+z = \text{constant} \}.$
\end{example}

\subsection{Poisson Hamiltonian integrators.}

As in classical numerical analysis, we call \emph{integrator} of order $k$
for a differential equation 
 $$ \dot{x}(t) = F(x(t))$$
a family of diffeomorphisms\footnote{There is a subtle point here: we cannot assume $\phi_h $ to be a well-defined diffeomorphism from $\mathcal U  $  to $\mathcal U $ for all $h$ small enough, but we can assume that for all $h $ small enough, there is an open subset $\mathcal U_h \subset \mathcal U $ on which $\phi_h $ is a diffeomorphism onto its image.
}  
$$\phi_h \colon \mathcal U  \longrightarrow \mathcal U,$$
    depending smoothly on a real parameter  $h $ such that the exact solution of the differential equation  coincides
    with  $\phi_h(x)$ up to order $k$ in $h$, i.e. $ \| \phi_h(x)- x(h) \| = \smallO{h^k}$. 
The \emph{numerical scheme of timestep $\Delta t$} associated to an integrator consists in
the recursive sequence $$x_0 = x \hbox{ and } x_{n+1} = \phi_{\Delta t}(x_n). $$

\begin{remark}
    Since a numerical scheme is defined by its iterations, the words \textit{integrator}, \textit{numerical method} and \textit{numerical scheme} can be understood without ambiguity as synonyms all along this article.
\end{remark}

Consider now a differential equation \eqref{eq:ham-inv} on $ \mathcal U \subset \mathbb R^n$ which is Hamiltonian for a Poisson structure $\pi $ and a Hamiltonian $H \in C^\infty(\mathcal U)$.

\begin{definition}

An integrator of order $k$ 
for the Hamiltonian differential equation \eqref{eq:ham-inv} 
$$\phi_h \colon \mathcal U  \longrightarrow \mathcal U,$$ is 
 said to be a \emph{Poisson integrator} if
        $\phi_h $ is a Poisson diffeomorphism of $(\mathcal U,\pi) $ for all $h$ for which it is defined\footnote{Again, it is more rigorous to say $\phi_h \colon \mathcal U_h \to \phi_h(\mathcal U_h) $ is a Poisson diffeomorphism}. 
\end{definition}

As explained in the Example \ref{poisnoham}, Poisson integrators can have a flow: the trajectories may jump from one symplectic leaf to another, and thus have non-physical behaviour. Hence, we formulate the following definition. 

\begin{definition}
\label{def:HamPoissonInt}
A Poisson integrator of order $k$
for \eqref{eq:ham-inv} 
$$\phi_h \colon \mathcal U  \longrightarrow \mathcal U,$$is 
 said to be a  \emph{Poisson Hamiltonian integrator of order $k$ for \eqref{eq:HamiltonianDiffEquation}} if 
  there exists a time-dependent Hamiltonian function $(H_t)_t$, depending smoothly on $t$, that coincides with $H$ up to order $k$, i.e. $H_t-H = \smallO{h^k}$
 whose integral curve coincides with the curve $h \longrightarrow \phi_h(x)$.
\end{definition}
\noindent

The following proposition claims that this second definition is strictly stronger.

\begin{proposition}
A Poisson Hamiltonian integrator is a Poisson integrator. 
\end{proposition}
\begin{proof}
The flow of a time dependent Hamiltonian differential equation is a Poisson diffeomorphism, as long as it is well-defined.
Also, if $H_t $ and $H$ coincide up to order $k$, then so do their Hamiltonian flows.
\end{proof}

\begin{remark}
\normalfont
As mentioned in the introduction, in general, we will not need in Definition \ref{def:HamPoissonInt} to describe explicitely the family $(H_t) $. All what matters at this point is that it exists.
\end{remark}

\section{Bi-realisations for Poisson manifolds}
\label{sec:PHI}

Having set the preliminaries and the framework in the previous sections, we are now ready to address the core of the paper: construct Poisson Hamiltonian integrators for a wide class of Poisson structures and any Hamiltonian differential equation on them. One of the important "tool" for the procedure is the notion of local symplectic groupoid associated to a Poisson structure that arose in \cite{Weinstein1987}. See \cite{Crainic2021} for a modern introduction to the matter. Notice that we will mostly not need the whole groupoid structure but a neighborhood of the identity of the latter, which can be considered to be closer from  the version of Karasev \cite{Karasev1987}. Several authors \cite{Ge1990}, \cite{deLeon2017} (to cite a few) have already used studied symplectic groupoids\footnote{A Lie groupoid over a manifold, roughly speaking, is a ``higher'' analog of a Lie group, where to each element one associate two mappings to this base manifold: source $\alpha$ and target $\beta$. Then two elements are composable when the source of one matches the target of the other, and for those the standard group axioms are satisfied. The symplectic form is compatible with this composition. Details are for example in \cite{daSilva1999}.} 
to construct numerical integrators: the relation is explained in \cite{oscar}.

Symplectic groupoids of Poisson manifolds are neither easy to understand as a notion, nor easy to construct as an object. Although we also somehow follow the same path, our method does not use the Lie groupoid structures (product, inverse) but only the source and target, so that all we need is what we call a \emph{bi-realisation}, so we will not have to define the notion in full generality.
 
 In the first subsections of what follows, we explain under which circumstances this bi-realisation, whose existence is guaranteed by theoretical arguments, is explicitly constructable. Then, we explain why the graph of $\epsilon dH $ gives a decent Poisson integrator for the differential equation \eqref{eq:ham-inv}.
 It is moreover possible to get a better Poisson integrator at an arbitrary order by replacing $H$ by a polynomial in $\epsilon$ whose terms are computed by an easy recursion, solving Hamilton-Jacobi equation at the desired order. Details are developed at the beginning of section \ref{sec:PHI-constr}. The modified Hamiltonian is also computed by recursion.  

While the justification of existence and estimates are guaranteed by complicated mathematical theorems, for implementation this section is sufficient and it does not assume advanced differential geometry knowledge; we again orient an interested reader to \cite{oscar} for details. So throughout the presentation in this section we will systematically make remarks on what is computable and with what precision.

\subsection{Bi-realisation I: Definition and existence.}

Assume we are given a \emph{bi-surjection}, i.e. the following data:
\begin{enumerate}
   \item an open subset $\mathcal U \subset \mathbb R^n $ -- the \emph{phase space}
    \item a subset $ \mathcal W \subset  \mathcal  U \times \mathbb R^n $ containing $ \mathcal U \times \{0\}  $, \\
    \item two surjective submersions, called \emph{source} and \emph{target}, $  \alpha, \beta \colon \mathcal W \to \mathcal U$ such that
     for all $x \in \mathcal U $
     $$ \alpha(x,0) = \beta(x,0) = x .$$ 
    
\end{enumerate}
We denote a bi-surjection by $\mathcal W \toto \mathcal U $.

Bi-surjections allow to recover a diffeomorphism of $\mathcal U $  out of any
\emph{bi-section}, i.e. any submanifold $\Sigma $ of dimension $n$ in $\mathcal W $ to which
the restrictions of the source $\alpha$ and the target $\beta$ are diffeomorphisms onto $\mathcal U$. (Sometimes, we only assume that $\alpha$ and $\beta$  are diffeomorphisms onto their images, which are open subsets of $\mathcal U$.) A bi-section $ \Sigma$ of a bi-surjection $\mathcal W \toto \mathcal U   $ induces a diffeomorphism $\underline{\Sigma} \colon \mathcal U \longrightarrow \mathcal U$ defined as $\beta \circ (\alpha_{|_{\Sigma}})^{-1} $:
   $$  \xymatrix{ & \mathcal W \ar[ddl]_\alpha \ar[ddr]^\beta & \\ &\ar@{^{(}->}[u] \Sigma \ar[dl]^{\simeq} \ar[dr]_{\simeq} & \\ \mathcal U \ar[rr]^{\underline{\Sigma}} & & \mathcal U }  $$
The crucial remark is that, if $ \Sigma$ and 
$\mathcal W \toto \mathcal U   $ are explicitly known, then the computation $\underline{\Sigma}  $ only requires to invert a diffeomorphism. This operation, in general, can be done numerically with machine precision and with reasonable cost, so that the diffeomorphism $ \underline{\Sigma}$ can be easily computed. 
The discretisations that we are going to construct
are families $(\underline{\Sigma}_h) $ of diffeomorphisms, depending on a ``small'' real parameter $h$, associated to a family $\Sigma_h $ of bi-sections such that $\Sigma_0 =  \mathcal U \times \{0\} $, so that $ \underline{\Sigma}_0$ is the identity map. 

Then $\mathcal W \subset \mathbb R^n \times \mathbb R^n$ comes equipped with a symplectic structure: 
 \begin{equation} \label{eq:omegaCan} \omega_{can} := \sum_{i=1}^n dp_i \wedge dx_i  
 \end{equation}
 with $x_1, \dots, x_n, p_1, \dots, p_n$ being the natural variables on $ \mathcal U \times \mathbb R^n $, labeled in that order. The corresponding Poisson structure satisfies:
  $$\{x_i,x_j\}_{\omega} = \{p_i,p_j\}_{\omega} =0 \hbox{ and } \{x_i,p_j\}_{\omega}  =\delta_i^j .$$
We can now state the main definition: 

\begin{definition}[Bi-realisation]
\label{def:bireal}
Let $\pi$ be a Poisson structure on an open subset  $\mathcal U  \subset \mathbb{R}^n$. 
A bi-realisation of $(\mathcal U,\pi) $
 is a bi-surjection $\mathcal W \toto \mathcal U $, with source $\alpha $ and target $\beta $, for which the auxiliary dimension coincides with the dimension of $\mathcal U $,  
satisfying the following:
 \begin{enumerate}
     \item $\alpha $ is a Poisson map,
     \item $\beta $ is an anti-Poisson map,
     \item the fibers of $\alpha$ and $\beta$ are symplectically orthogonal to each other.
 \end{enumerate}
In all three items above,
$\mathcal{W}$ comes equipped with the Poisson bracket $\{\cdot ,\cdot \}_{\omega} $ associated to the symplectic structure \eqref{eq:omegaCan}. 
\end{definition}

\begin{remark}
\normalfont
Conditions 1 -- 3 in Definition \ref{def:bireal}
mean that for all functions $F,G \in C^\infty(\mathcal U)$:
 $$  \alpha^* \{F,G\}_\pi = \{\alpha^*F,\alpha^* G\}_\omega, \quad \beta^* \{F,G\}_\pi = -\{\beta^*F,\beta^* G\}_\omega,
 $$ $$\{\alpha^*F,\beta^* G\}_\omega=0.$$
\end{remark}

We will quote the following two results (without proof), for completness and future references.

\begin{proposition} \label{prop:bi-real-poisson}
Let $\mathcal W \toto \mathcal U $ be a bi-realisation. For any bi-section $\Sigma  \subset \mathcal W$ which is Lagrangian with respect to \eqref{eq:omegaCan}, the induced diffeomorphism $\underline{\Sigma}\colon \mathcal U \longrightarrow \mathcal U$ is a Poisson diffeomorphism. \\
Moreover, for any Lagrangian submanifold of the form
 $$ \{(x_1, \ldots, x_n, \partial_{x_1}F, \ldots, \partial_{x_n}F), (x_1, \ldots, x_n) \in \mathcal{U}\} \subset \mathcal{W} $$
 for $F \in C^\infty(\mathcal U)$ a small enough\footnote{More precisely, the $\infty$-norm of the derivative at a point must be smaller than some local bound.} smooth function, this Poisson diffeomorphism is the value at time $1$ of the flow of a time-dependent Hamiltonian vector field.   
\end{proposition}

\begin{theorem}[Existence and uniqueness]
\label{theo:fromtheSky}
Any Poisson structure on an open subset $\mathcal  U \subset \mathbb R^n$  admits a bi-realisation. Furthermore, it is canonical in the following sense: two different bi-realisations above a Poisson structure are symplectomorphic through a unique symplectomorphism fixing  $\mathcal  U \times \{0\}.$
\end{theorem}

\subsection{Bi-realisations II: Explicit constructions.}

Theorem \ref{theo:fromtheSky}
states that bi-realisations do exist and are unique. Below we explain how one can construct them.

\subsubsection*{Cotangent paths}

Let $(M,\pi) $ be a Poisson manifold\footnote{By this we mean a collection of open sets $\mathcal U$ equipped with Poisson bivectors in a consistent way. The result being essentially local, one may think of just one open set as before, and in the applications we will work in one coordinate chart anyway.}. Out of a path $\alpha $ valued in $T^*M $, two paths valued in $TM$ can be constructed:
\begin{enumerate}
    \item consider $\dot{\gamma}(t) $ with $\gamma = \tau \circ \alpha $ the base path of $\alpha $, $\tau$ being the projection defining $T^*M$.
    \item  $\alpha(t) \in T^*_{\gamma(t)}M$ with  $\pi_{\gamma(t)} \in \wedge^2 T_{\gamma(t)} M$: the contraction $\pi^\#(\alpha) $ is a path valued in $TM$.
\end{enumerate}
We call a path \emph{cotangent} when both $TM$-valued paths above coincide.  

When an affine connection $\nabla $  is given on $T^*M $, every $\xi \in T^*M  $ is a starting point of a parallel cotangent path  $t \mapsto \xi(t) $ required to satisfy the additional condition: 
 $$  \nabla_{\dot{\gamma}(t)}  \xi = 0 .$$
There exists a neighborhood\footnote{For purpose of notation, we denote it $\mathcal W $ again, even though it is a collection of such open sets described before.} $\mathcal W $ of $M$ in $ T^* M$ for which the parallel cotangent path above is defined for all $ t \in [-1,1]$. We call \emph{geodesic flow} the map:
  $$ {\Xi}_{\nabla,\pi}: \xi \in \mathcal W \to \xi(t) \in \{ \hbox{Cotangent Paths} \} .$$

\subsubsection*{Karasev's construction}

Let $(M,\pi) $ be a Poisson structure with $M$ a subset of $ \mathbb R^n$,
so that $T^*M  $ can be identified with pairs $ m \in M $ and $\xi \in \mathbb R^n $. 
Let $\nabla $ be the trivial affine connection.
For every $\xi \in \mathbb R^n $, $  {\Xi}_{\nabla,pi} (m,\xi) $ defines a path $m(t)$ solution of the differential equation:
  $$ \dot{m_\xi}(t) = \pi_{m_\xi(t)}^\# \xi .$$
For every $\xi $, consider the open subsets $M^\alpha_\xi \subset M $ (resp. $M^\beta_\xi \subset M$) on which this path $m_\xi$ is well defined for all $t \in [0,1] $ (resp. all $t \in [-1,0] $).

The idea of Karasev consists in looking at the following two equations whose unknown $\alpha,\beta$ are in $M$, for a given $(m,\xi) \in T^* M$: 
$$ \int_{0}^1 \Xi ( \alpha , \xi ) (t) dt = m \hbox{ and } \int_{-1}^0 \Xi ( \beta , \xi ) (t) dt = m. $$
Since, for $\xi=0 $, the unique solutions are $ \alpha=\beta=m$, 
there exists a neighborhood $\mathcal{W}$ of $M$ in $T^* M $ on which the two previous equation have a unique solution, defining therefore two maps $\mathcal{W}\to M $ that we denote $\alpha $ and $\beta $. 

\begin{proposition}[Karasev]
The triple $(\mathcal{W},\alpha,\beta)$ is a bi-realisation for a Poisson structure $ (M,\pi)$.
\end{proposition}

\begin{remark}
This bi-realisation is explicit provided that the geodesic flow $\Xi$ and its integral can be computed. It is computable by quadrature if so is the geodesic flow, which is the case for a large class of Poisson structures.    
\end{remark}

\subsubsection*{Poisson Spray and Moser's trick}
A sligthly more academic (axiomatic) and thus conceptual approach to the above construction may be presented using the notion of Poisson spray.

\begin{definition}
Let $(M,\pi)$ be a Poisson manifold, $\tau \colon T^*M \to M$ the cotangent projection and for $\lambda \in \mathbb{R}^*,$ $m_\lambda \colon \xi \in T^*M \mapsto \lambda \xi \in T^*M$ the fiberwise multiplication by $\lambda.$ $X \in \mathfrak X (T^*M)$ is said to be a Poisson spray if it verifies the following two conditions:
\begin{enumerate}
    \item $\forall \xi \in T^*M,$ $\text{d}_\xi \tau .X(\xi) = \pi^{\#}(\xi),$
    \item $X$ is homogeneous of degree 1: $\forall \lambda \in \mathbb{R}^*,$ $\text{m}_\lambda^*X = \lambda X,$ \\
    i.e. $\text{d}_{\alpha \xi} m_{\lambda^{-1}} . X(\alpha \xi) = \lambda X(\xi).$
\end{enumerate}
\end{definition}

\begin{example}
For some choice of coordinates $x_i$, the Poisson tensor has the form $$
\pi(x) = \sum_{1 \leq i<j \leq n} \pi_{ij}(x) \partial_{x_i}\wedge \partial_{x_j}.$$ 
Denoting $(x,\xi)$ the induced cotangent coordinates,  
$$X = \sum_{1 \leq i<j \leq n} \pi_{ij}(x) \xi_j \partial_{x_i}$$ is a Poisson spray.  
\end{example}

The second point of its definition implies that $X$ vanishes on the zero section $0_{T^*M}$. Consequently, there exists a neighborhood $\mathcal{W}$ of $0_{T^*M}$  such that the time-1 flow of $X$ $\Phi^1_X \colon U \to \Phi^1_X(U)$ is a well-defined global diffeomorphism onto its image.

\begin{remark}
For a given Poisson structure, Poisson sprays always exist (see \cite{Crainic2011}). However, Poisson sprays are far from being unique. For instance, one can add a term of the form "$f(x)\xi_i \xi_j \partial_{\xi_j}$" to it -- this is an important freedom that allows to construct explicit integration of the flow above in a lot of important cases.
\end{remark}

\begin{theorem}\label{symp_grpd}
Any Poisson spray induces target, source and multiplicative form of the local symplectic groupoid near $0_{T^*M}$ in the following way :
\begin{enumerate}
    \item $\bar \alpha  = \tau \colon T^*M \to M,$
    \item $\bar \beta = \tau \circ \Phi^1_X,$
    \item $\Omega = \int_0^1 \Phi^{\,s}_X {}^* \omega \, \text{d}s,$ where $\omega$ is the canonical symplectic form.
\end{enumerate}
\end{theorem}

Note that $\Omega$ is symplectic up to shrinking of $\mathcal{W}.$

\begin{theorem}
Any Poisson spray induces a bi-realisation.
\end{theorem}

\begin{proof}
By Moser's trick, $\omega$ and $\Omega$ are symplectomorphic in a neighborhood of the zero section in $T^*M$: 
$ \omega = \psi^* \Omega.$ Moreover, $\psi$ is the identity map on $M \subset T^*M$, and so is its differential at any point of $M$.
Then, a bi-realisation on this neighborhood is given by :
$    \alpha = \psi^* \bar \alpha, \quad    \beta = \psi^* \bar \beta.$
\end{proof}

\noindent
However, for a generic Poisson structure, the Poisson spray, and its flow $\psi$, may not be explicitly computable.

\subsection{Examples}  \label{sec:poisson-examples}
\;
In what follows we construct a symplectic bi-realisation for several classes of Poisson structures, using various techniques, including Poisson sprays. We start with the simplest Poisson structure given by a symplectic form written in canonical (Darboux) coordinates, to recover some symplectic integrators.
Then we continue with a couple of constructions that will later be used in the numerical tests.

\subsubsection*{Symplectic case.} 
\; \\
Let $M = \mathbb{R}^{2n} = \{(q,p)\}$, then $\pi = \partial_p \wedge \partial_q.$ Let $(q,p,\xi_q, \xi_p)$ be cotangent coordinates on $R^{4n}.$ A Poisson spray is $X(q,p,\xi_q, \xi_p) = \xi_p \partial_q - \xi_q \partial_p.$ \\
The objects of theorem \ref{symp_grpd} are: 
\begin{enumerate}
    \item $\bar \alpha \colon T^*M \to M \colon q,p,\xi_q, \xi_p \mapsto (q,p),$
    \item $\bar \beta \colon T^*M \to M \colon q,p,\xi_q, \xi_p \mapsto (q+ \xi_p, p - \xi_q),$
    \item $\Omega = \omega + \frac{1}{2} \text{d}p  \wedge \text{d}\xi_p  - \frac{1}{2} \text{d}\xi_q \wedge \text{d}q - \frac{1}{3} \text{d}\xi_q \wedge \text{d}\xi_p$.
\end{enumerate}
The symplectomorphism between $\Omega$ and $\omega$ is given by
$$
\Psi \colon 
    \begin{pmatrix}
    q\\
    p\\
    \xi_q\\
    \xi_p
    \end{pmatrix} \mapsto \begin{pmatrix}
    q - \frac{\xi_p}{2}\\
    p + \frac{\xi_q}{2} \\
    \xi_q \\
    \xi_p
    \end{pmatrix},
$$
and the resulting bi-realisation is 
$$
    \left\{
    \begin{array}{ll}
        \alpha: (q,p,\xi_q,\xi_p) \mapsto (q - \frac{1}{2} \xi_p, p + \frac{1}{2} \xi_q)\\
        \beta: (q,p,\xi_q,\xi_p) \mapsto (q + \frac{1}{2}\xi_p, p - \frac{1}{2}\xi_q)
    \end{array}
    \right. .
$$

\subsubsection*{Quadratic Poisson structures} 
\; \\
The following example will be important for Lotka-Volterra systems. Consider $M = \mathbb{R}^n$ and a quadratic Poisson structure:
\begin{equation} \label{eq:quad-Poisson}
\pi = \sum_{1 \leq i,j \leq n} a_{ij} x_i x_j \partial_{x_i} \wedge \partial_{x_j}    
\end{equation}
 Using the (natural) Poisson spray of \cite{Li2018}:
$$
    X = \sum_{1 \leq i,j \leq n} a_{ij} x_i x_j \xi_i \partial_{x_j} - \sum_{1 \leq i,j \leq n} a_{ij} x_i \xi_i \xi_j \partial_{\xi_j},
$$
and the symplectomorphism:
$$
\Psi \colon 
    \begin{pmatrix}
    x_j\\
    p_j
    \end{pmatrix} \mapsto \begin{pmatrix}
    e^{-\frac{1}{2}\sum_i a_{ij}x_i p_i} x_j \\
    e^{\frac{1}{2}\sum_i a_{ij}x_i p_i} p_j
    \end{pmatrix},
$$
one constructs the following global bi-realisation:
$$
    \left\{
    \begin{array}{ll}
        \alpha: (x,p) \mapsto \left( e^{-\frac{1}{2} \sum_i a_{ij} x_i p_i }. x_j \right)_{j=1, \dots, n}\\
        \beta: (x,p) \mapsto \left( e^{\frac{1}{2} \sum_i a_{ij} x_i p_i }. x_j \right)_{j=1, \dots, n}
    \end{array}
    \right. .
$$

\subsubsection*{Dual of a Lie algebra: cotangent lifts} \; \\
In the case of the linear Poisson structure on the dual of a Lie algebra, there is another way of constructing bi-realisations.
\begin{proposition} \label{Lie-Poisson}
Let  $\mathfrak{g}$ be a Lie algebra of a Lie group $G$, and
$\varphi \colon G \to \mathfrak{g}$ a local diffeomorphism in $1_G$, bijective on an open subset $\mathcal V$ containing the unit $1_G$, such that :
\begin{enumerate}
    \item $\varphi(1_G) = 0$
    \item $T_{1_G} \varphi = Id.$
\end{enumerate}
Let us denote by $\psi$ the inverse of $\varphi$ and $\bar{\mathcal{V}} = \varphi(\mathcal{V}).$ \\
Then a bi-realisation of the Lie-Poisson structure on $\mathfrak{g}^* $ is given by:
$$
    \left\{
    \begin{array}{ll}
        \alpha \colon \bar{\mathcal{V}} \times \mathfrak{g}^* \to \mathfrak{g}^*: (\eta,\xi) \mapsto R^*_{{\psi(\eta)}} \left (T_{\psi(\eta)}^*\varphi \right).\xi\\
        \beta \colon \bar{\mathcal{V}} \times \mathfrak{g}^* \to \mathfrak{g}^*: (\eta,\xi) \mapsto L^*_{{\psi(\eta)}} \left (T_{\psi(\eta)}^*\varphi\right).\xi =\text{Ad}^*_{\psi(\eta)}.\alpha(\eta,\xi)
    \end{array}
    \right. .
$$
\end{proposition}

Let us describe more precisely these source and target maps. They are the dual of the inverse of the so-called right and left logarithmic derivatives of $\psi$. Since $\varphi $ maps $ \mathcal{V} \subset G$ to $\bar{\mathcal{V}} \subset \mathfrak{g}$, its differential $T\varphi $ maps $T\mathcal V$ to $T \bar{\mathcal{V}} \simeq \bar{\mathcal{V}} \times \mathfrak g$.
Composing this map with the right and left identifications of $T \mathcal{V}$ with $\mathcal{V} \times \mathfrak g$ and using the diffeomorphism $\psi,$ one gets two families indexed by $\eta \in \bar{\mathcal{V}}$ of linear invertible endomorphisms of $\mathfrak g.$
The source and targets above are the dual of these maps.\footnote{The reader familiar with the notion of \textit{logarithmic derivative} will notice that those maps are the inverse of the dual of the logarithmic derivative of $\psi$ after right and left trivialisations of $TG.$}

\begin{remark}
Notice that we no not assume $\varphi $ to be the logarithm, i.e. the inverse of the exponential map. It may be any local diffeomorphism. In fact, the logarithm map may not be a good choice since its differential may be too complicated to compute.
\end{remark}

\begin{remark}
If  $ \varphi$ and its inverse are explicitly computable, then so are $\alpha $ and $\beta $.  
\end{remark}

\begin{example}
\normalfont
Let us spell-out this construction in the case of the algebra $so(n)$ of anti-symmetric matrices. The scalar product $<.,.> \colon (X,Y) \in so(n) \mapsto \text{Tr}(X^T.Y) \in \mathbb{R}$ induces an isomorphism between $so(n)$ and its dual. The local diffeomorphism we use is
$$
    \varphi \colon SO(n)_+ \to so(n) : Q \mapsto 4 \frac{Q-I}{Q + I}
$$
with the inverse
$$
    \psi \colon so(n) \to SO(n)_+ : A \mapsto  \frac{4 + A}{4 - A}.
$$
Its derivative is
$$
T_Q \varphi \colon so(n) \to so(n) : H \mapsto 4(I + Q^{-1})^{-1}.H.(I+Q)^{-1}
$$
and the transpose of it by $<.,.>$ is its cotangent lift:
$$
\begin{array}{cccc}
    T^* \varphi  \colon & so(n) \times so(n)^* & \to &SO(n) \times so(n) \\
    &(H,x)& \mapsto &(\psi(H), (I+\frac{A}{4}).x.(I-\frac{A}{4}).
\end{array}
$$
Since the metric is Ad-invariant, $\text{Ad}^*_Q x = Q^{-1}.x.Q $. 
And the source and target are:
$$
    \left\{
    \begin{array}{ll}
        \alpha \colon so(n) \times so(n) \to so(n): (A, x) \mapsto (1+\frac{A}{4}).x.(1-\frac{A}{4}) \\
        \beta \colon so(n) \times so(n) \to so(n) \colon (A, x) \mapsto (1-\frac{A}{4}).x.(1+\frac{A}{4})
    \end{array}
    \right. .
$$
\end{example}

\section{Explicit construction of Poisson Hamiltonian integrators}
\label{sec:PHI-constr}

We are now ready to put together all what has been discussed in the context of Poisson geometry in the two previous sections, and make the final step to construction of the appropriate structure preserving integrators. To sum it up, we start with a Poisson structure $\pi $ defined on an open subset $\mathcal U \subset \mathbb R^n $. The only assumption that we need it that it admits an \emph{explicit} bi-realisation $\mathcal W \toto \mathcal U$.

 We recall that $\mathcal W$ is an open subset of $\mathcal U \times \mathbb R^n$ containing $\mathcal U \times \{0\} $. 
 We denote its source by $\alpha $, its target by $\beta$, and its base map by $\tau $. We denote by ${\mathbf{0}}: \mathcal U \to \mathcal W $ the map ${\mathbf{0}} (x)=(x, 0)  $.
  $$ \xymatrix{ &\mathcal W \ar@/^/[d]^{\tau} \ar[dl]_{\alpha} \ar[dr]^{\beta} & \\ \mathcal U & \ar@/^/[u]^{{\mathbf{0}}} \mathcal U& \mathcal U}  $$
 
 \begin{remark}
 \normalfont
 We recall that for any $(x, p) \in \mathcal W$, $\alpha(x, p)  $ and $\beta(x, p) $ are in the same symplectic leaf of $\pi $. This leaf is not the same symplectic leaf at the one containing $\tau(x, p)=x $.
 \end{remark}

\noindent
Consider again the Hamiltonian differential equation 
\begin{equation}
     \label{eq:hamH}\dot{x}(t) = \pi ( x(t)) (\nabla H(x(t))) 
\end{equation}
for some Hamiltonian $ H \in C^\infty(\mathcal U)$.

We claim that we can construct an explicit Poisson Hamiltonian integrator of order $k$ for
\eqref{eq:hamH}. There are several steps that we now present.

\begin{enumerate}
    \item[Step 1.] To start with, one needs to compute the first $k$  terms of the 
\emph{Hamilton-Jacobi transform} of $H$. The latter is a formal series of smooth functions on $\mathcal U $ of the form 
    $$ \mathcal S_h (H)(x) = h S_1(x) + \frac{h^2}{2} S_2(x)   + \frac{h^{3}}{6} S_3 (x) + \dots ,$$
    whose coefficients are computed by recursion as follows:
    \begin{enumerate}
        \item  Set $S_1(x) = H(x) $.
        \item[] (In particular,  for $k=1$, the truncation of the generating transform of $H(x)$ is simply  $ h H(x) $.)
        \item The smooth function $S_{k+1}(x) $ is then given by the recursive formula:       
$$
S_{i+1}(m) = \left. \frac{d^i}{dt^i}\right|_{t=0} H\left( \alpha\left(d_m S_t^{(i)} \right)\right)
$$
where we write $S_t^{(i)} =\sum_{j=1}^i \frac{t^j}{j!} S_j.$
          \item[] Since the bi-realisation is supposed to be explicitly known, the construction of these terms can be done explicitly as well.
    \end{enumerate}
    
    \item[Step 2.]
    Now starts the construction of the Poisson Hamiltonian  integrator itself. Choose a timestep, i.e. fix a small positive real number $h$. 
We define a numerical scheme approximating the integral curve of \eqref{eq:hamH} with initial value $x_0 $ by constructing the sequence $ (x_n)$ according to the following recursion:
\begin{enumerate}
    \item  Assume that for every $n \in \mathbb N$,
 the equation $$  \alpha \left (y_n, \sum_{i=1}^k h^i \, \nabla S_i (y_n) )\right) = x_n $$ admits a unique solution $y_n$ (otherwise, it means that the time step is too large). 
\item 
     Set
      $$ x_{n+1} := \beta \left(y_n, \sum_{i=1}^k h^i \, \nabla S_i (y_n) \right).$$
\end{enumerate}
\end{enumerate}

\begin{remark} \label{ca}
    The computations related to formal power series in Step 1. can be done efficiently with computer algebra tools. The resolution of the implicit relation in Step 2 is done approximately (for example by fixed point techniques), but can eventually be done with machine precision.
\end{remark}

\begin{example}
\normalfont
For $k=1$, this numerical scheme consists in mapping $x_n $ to $\beta(y_n, h \nabla H(y_n)) $ where $y_n$ is the unique solution of $\alpha (y_n, \nabla H(y_n))=x_n $.
\end{example}

\begin{example}
In the case of a Lie-Poisson structure on a Lie algebra $\mathfrak g $ equipped with a local diffeomorphism $\varphi \colon G \to \mathfrak g $ with inverse $\psi \colon  \mathfrak g \to G$, for $k=1$, our Hamiltonian Poisson integrator consists in 
\begin{enumerate}
\item Compute $a \mapsto \left( \psi(\eta) \right)^{-1} T_\eta \psi (a)$\footnote{where $\left( \psi(\eta) \right)^{-1}$ is the inverse of $\psi(\eta)$ for the group law of $G$}, which is a family depending on $\eta \in \bar{\mathcal{V}}$ of diffeomorphisms $  \mathfrak g \simeq \mathfrak g $. Then consider the dual  of its inverse, which is now a family of maps $D_{\eta} \colon \mathfrak g^* \simeq \mathfrak g^*$ being the identity map for $\eta =0 $. Then  solve $ D_{\Delta t \nabla H(y_n)} (y_n)  = x_n   $.
\item Consider
$$x_{n+1} = {\mathrm{Ad^*}}_{\psi(\Delta t \nabla y_n H)}  x_n $$
\end{enumerate}

\end{example}

\begin{remark}
\normalfont 
By construction, $x_n $ and $ x_{n+1}$ belong to the same symplectic leaf. But to go from $x_n $ to $x_{n+1}  $ one uses a point $y_n $ which is not, in general, on that common symplectic leaf. This is extremely counter-intuitive. 
\end{remark}

\begin{remark}
\normalfont
Let us give the first terms of the Hamilton-Jacobi transform:
\begin{enumerate}
    \item $S_1 = H$
    \item $S_2 ={\mathbf{0}}^*\left( \frac{1}{2}  \{\alpha^*H, \tau^*S_1 \} \right)$
    \item 
$ S_3 = {\mathbf{0}}^* \left( \frac{1}{3}\{ \alpha^*H, \tau^*S_2 \} + \frac{1}{6} \{ \{\alpha^*H, \tau^*S_1 \}, \tau^* S_1 \}\right)$
\item $ S_4 = {\mathbf{0}}^*\left( \frac{1}{4}\{\alpha^*H, \tau^*S_3 \} + \frac{1}{3!} \{ \{ \alpha^*H, \tau^*S_2 \}, \tau^*S_1 \} + \right.$ \\ $\left. + \frac{1}{4!} \{ \{ \{\alpha^*H, \tau^*S_1 \}, \tau^*S_1 \}, \tau^*S_1\}\right)
 $
\end{enumerate}
Above, ${\mathbf{0}}^* $ means that the function on $\mathcal W$ is restricted to $\mathcal U\times \{0\}$, and therefore considered as a function on $\mathcal U$.

\end{remark}

\begin{theorem}
\label{theo:explicit}
The above numerical scheme defines a Hamiltonian Poisson integrator at order $k$ for the Hamiltonian differential equation \eqref{eq:hamH}.
\end{theorem}

\begin{proof}(See \cite{oscar})
Let us give a brief outline of the proof of this theorem, which will also explain how the
time-dependent Hamiltonian $H_t $ whose flow at time $h$ matches exactly $ x_n$ to $ x_{n+1}$ is constructed.

To start with, recall the two points that we saw in Section \ref{sec:PHI} about the set
 $Gr(dS) := \left\{(x, d_x S) , x \in \mathcal U\right\} $ 
 for $S \in C^\infty(\mathcal U)$ 
\begin{enumerate}
    \item It is a Lagrangian subset of $\mathcal U \times \mathbb R^n $.
    \item Provided that the differential of the function $S$ is small enough, it is a bisection\footnote{Meaning that restrictions of $\alpha$ and $\beta$ to $Gr(dS)$ are invertible.} of $\mathcal W \toto \mathcal U$.
\end{enumerate}
As a consequence, as we saw in Proposition \ref{prop:bi-real-poisson}, the map $\beta \circ \alpha^{-1}|_{Gr(dS)} $ is a Poisson map.

It is a more subtle result that if the function $S$ has a small enough differential, then the Poisson morphism is the time $1$-flow for a time dependent Hamiltonian function. For instance, if $S$ depends on a parameter $h $, i.e. $S= S_h $ where $ S_t$ is a time dependent function with $S_0=0$, then this time dependent Hamiltonian function is given by:
 $$ \tilde{S}_t (x):= \frac{\partial S_t}{\partial t}(y)   $$
where $ y $ is chosen such that $(y, \nabla S_t(y)) \in Gr(dS_t) $. 
Afterwards, the question reduces to finding $S_t$ such that the flow of $\tilde{S}_t$ at time $h$ matches the flow of $H$ at time $h$ up to order $k$ in the variable $h$, as in Step 2.
\end{proof}

\begin{remark}
At order $k=1$ in the symplectic case (i.e. non-degenerate constant Poisson structure), it is easy to check that for the harmonic oscillator $H = \frac{p^2 + q^2}{2}$, one recovers the symplectic mid-point scheme. For a general Hamiltonian $H,$ the present construction gives the fact that an implicit Euler scheme of timestep $\frac{\Delta t}{2}$ composed with an explicit Euler scheme of timestep $\frac{\Delta t}{2}$ is a symplectic integrator of order $1$ and timestep $\Delta t.$ 
More generally, for higher orders the constructed Poisson Hamiltonian integrators for symplectic structures will be symplectic integrators, but a priori different from the standard symplectic Runge-Kutta methods.
\end{remark}

\begin{remark}
We have mentioned in the introduction that the naive idea  ``restrict to a leaf, be symplectic there'' to recover Poisson globally, does not work because is almost never constructive. 
But the other way around it is actually fruitful: now having constructed a Poisson Hamiltonian integrator forcing the trajectory to stay on the correct leaf, one can apply the backward analysis techniques (restricted to leaves) for error estimates.
\end{remark}

\begin{remark}
    Recall that in the case of linear Poisson structures of Proposition \ref{Lie-Poisson}, the construction of the bi-realisation amounts to computation of the coadjoint action of $G$ on $\mathfrak g$, and construction of 
a local diffeomorphism:  $ \phi \colon G \to \mathfrak g$ with its differential at $1$ being the identity.    

The obtained Hamiltonian Poisson integrator of order $1$ is of the form:
 $$ x_{n+1} = {\mathrm{Ad^*}}_{\phi^{-1}(\Delta t \nabla y_n H)}  x_n $$
 which is certainly not surprising: any such numerical scheme stays in the symplectic leaf where one starts from. The same remark about the point $y_n$ outside this leaf holds. 

An obvious choice for $\phi $ is the inverse of the exponential map, but there is some freedom in it: any such a local diffeomorphism can be used to compute an Hamiltonian Poisson integrator up to order $k$. It is important, however, to be able to compute easily its differential and its inverse.
\end{remark}

\section{Numerical tests}
\label{sec:num}
In this last section we illustrate the advantages of Poisson Hamiltonian integrators on a couple of examples.

\subsection{The Rigid Body}

First turn to the linear Poisson structures -- a good example of those can be provided by the dynamics of a rigid body  about a periodic orbit.

The equations governing the system read: 
$$
  \dot x = -x \wedge J.x,
$$
where: $\wedge$ denotes the vector product in $\mathbb{R}^3,$ and  the symmetric positive matrix $J$ is the inertia tensor of the body.

It is a Hamiltonian differential equation for $\pi(x) = \begin{pmatrix}
    0 & -x_3 & x_2 \\
    x_3 & 0 & -x_1\\
    -x_2 & x_1 & 0
\end{pmatrix}$ and $H(x) = \frac{1}{2}\text{Tr}\Big(j(x)^T . J . j(x)\Big)$ where 
$j \colon         \mathbb{R}^3 \widetilde{\longrightarrow} so(3)$ given by:
$$
        x \longmapsto \begin{pmatrix}
    0 & -x_3 & x_2 \\
    x_3 & 0 & x_1\\
    -x_2 & -x_1 & 0
\end{pmatrix}.
$$
We consider the inertia tensor
$J = \begin{pmatrix}
 1 & 0 & 0\\
0 & \pi & 0\\
0 & 0 & 100
\end{pmatrix}$ and $x_0 = 
\begin{pmatrix}
    1\\
    1\\
    1
\end{pmatrix}$
so that the trajectory is given by Figure \ref{fig:rigid_traj}.
\begin{figure}[htp]
    \centering
    \includegraphics{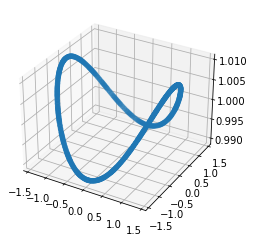}
    \caption{The trajectory of the angular velocity of a rigid body in $\mathbb{R}^3.$}
    \label{fig:rigid_traj}
\end{figure}
Numerical simulations are for timestep $\Delta t = 10^{-4}.$ The Poisson Hamiltonian integrator of order 2 behaves much better than the Runge-Kutta method of order 4 in the preservation of both Casimir and Hamiltonian levels (Figure \ref{fig:Rigid_body_ham_RK4_vs_phi2}). 

Details are a bit more involved. The error of the traditional method depends linearly on the number of iterations and so diverges from the continuous (closed) trajectory. A Poisson Hamiltonian integrator preserves Casimir level at machine precision and oscillates around a Hamiltonian value with the amplitude depending on $\Delta t^k,$ $k$ being the order of the method. We emphasize that this distance \textit{does not} depend on the amount of iteration. One recovers a typical stability phenomenon of symplectic integrators, already noticed and explained in \cite{Benettin1994}. A zoom on Hamiltonian errors is made in Figure \ref{fig:Rigid_body_ham_PHI2}. For Poisson Hamiltonian integrators, a theoretical explanation relies on the Magnus formula for Poisson structures introduced in \cite{oscar}. Those phenomena are illustrated on the schematic section of the trajectory -- Figure \ref{fig:Rigid_body_illustration}.

\begin{figure}[htp]
        \centering
        \begin{subfigure}[b]{\textwidth}
            \centering
            \includegraphics[width=\textwidth]{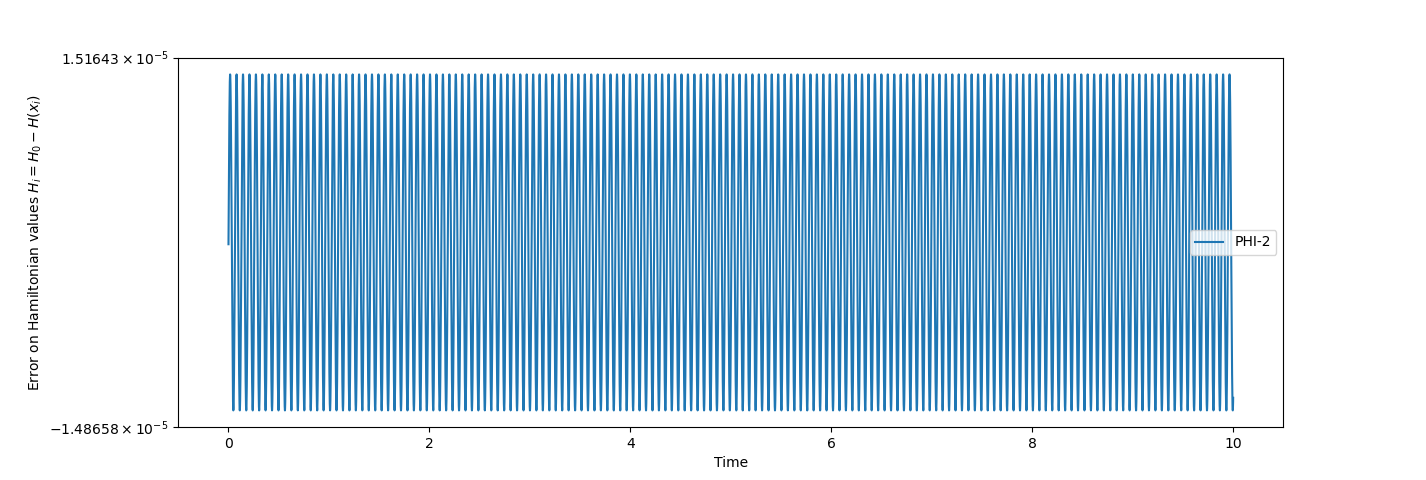}
            \caption{Errors on Hamiltonian values for PHI-2}
            \label{fig:Rigid_body_ham_PHI2}
        \end{subfigure}
        \begin{subfigure}[b]{0.475\textwidth}
            \centering
            \includegraphics[width=\textwidth]{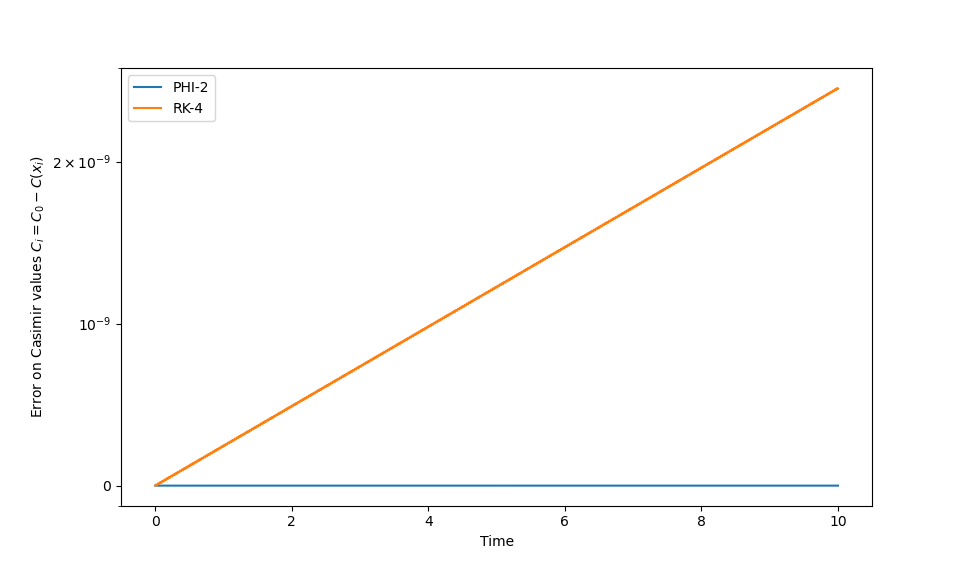}
            \caption{Error on Casimir values}
        \label{fig:Rigid_body_cas_RK4_vs_phi2}
        \end{subfigure}
        \hfill
        \begin{subfigure}[b]{0.475\textwidth}  
            \centering 
            \includegraphics[width=\textwidth]{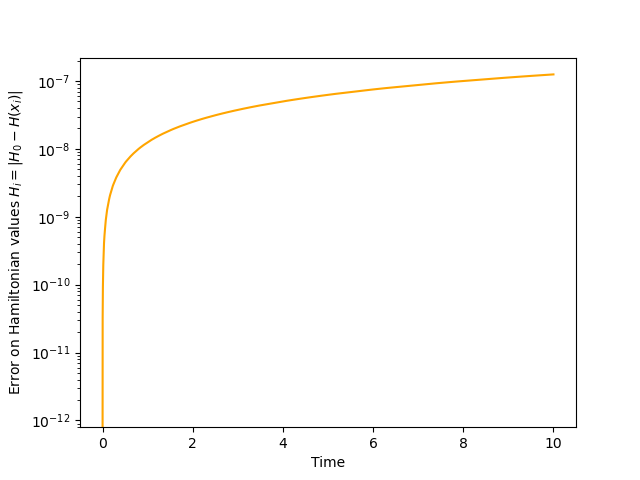}
            \caption{Errors on Hamiltonian values for RK-4}
        \end{subfigure}
        \caption{Comparison between Runge-Kutta 4 and our Poisson Hamiltonian integrator at order 2 for the Rigid Body dynamics}
        \label{fig:Rigid_body_ham_RK4_vs_phi2}
\end{figure}

\begin{remark}
    The Casimir is the square of the norm. Hence Figure \ref{fig:Rigid_body_cas_RK4_vs_phi2} indicates that RK-4 iterations will converge to $0$ in $\mathbb{R}^3,$ which is a fixed point of the dynamics as well as a singular leaf of the foliation of the total space. This lead on long run simulations to pathological behaviours. In turn, it stresses the importance of numerical methods preserving leaves of a singular foliation such as the ones appearing in Poisson structures.
\end{remark}

\begin{figure}[htp]\label{fig:Rigid_body_picture}
        \centering
        \begin{subfigure}[b]{0.475\textwidth}
            \centering
            \includegraphics[width=\textwidth]{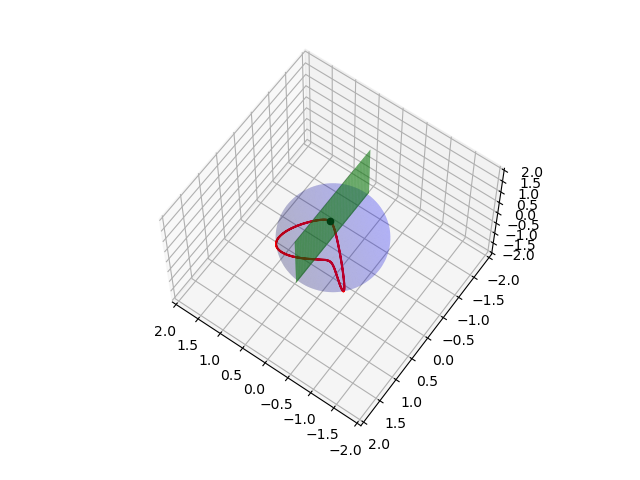}
            \caption{A Poincaré section of the trajectory (green plane)}
            \label{fig:Rigid_body_section}
        \end{subfigure}
        \hfill
        \begin{subfigure}[b]{0.475\textwidth}  
            \centering 
            \includegraphics[width=0.7\textwidth]{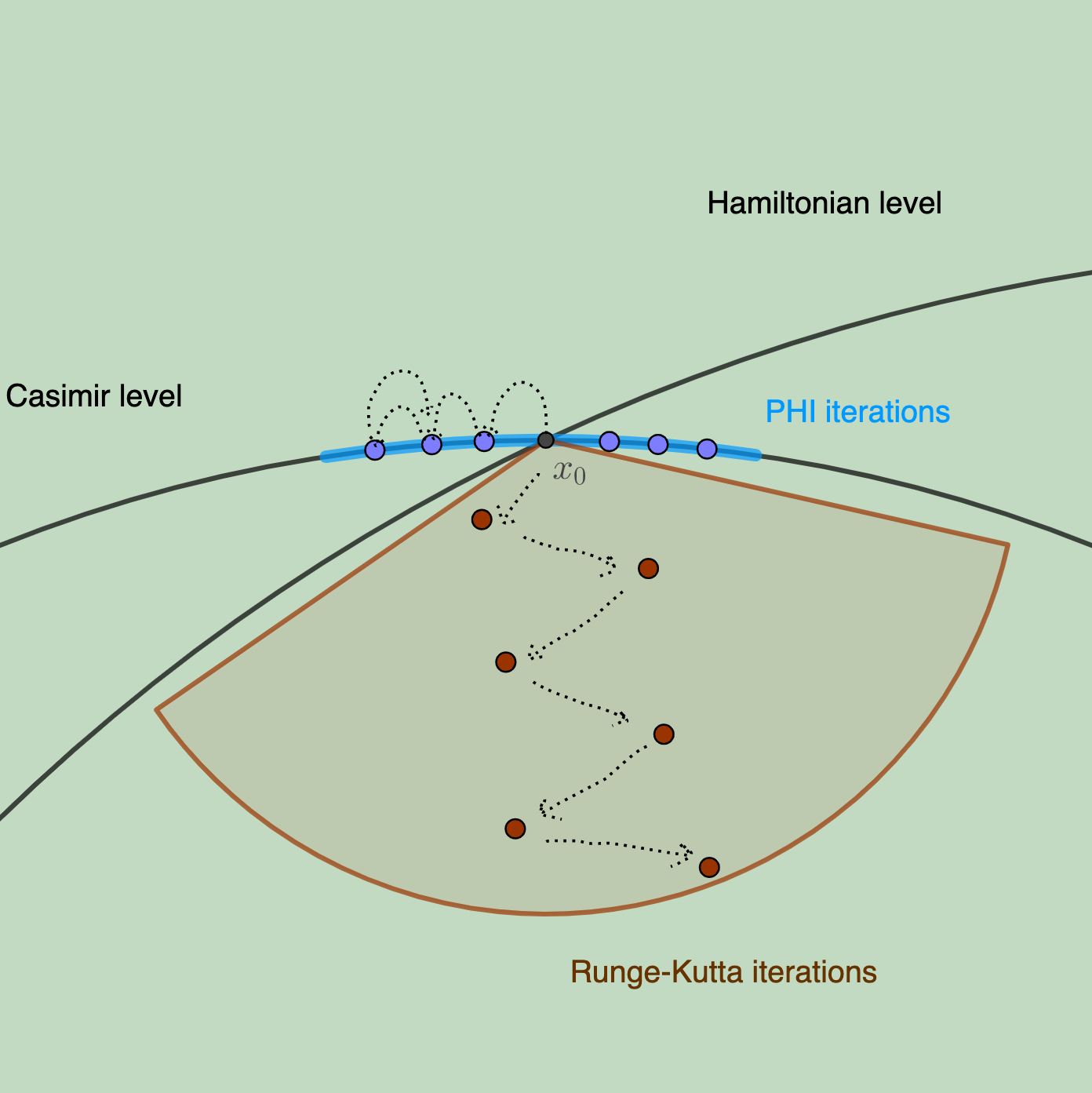}
            \caption{Intersection between (interpolated) discrete trajectories and the Poincaré section}
            \label{fig:Rigid_body_iterations_on_section}
        \end{subfigure}
        \caption{Illustration of the difference of behaviour of two numerical methods with respect to Hamiltonian and Casimir levels}
        \label{fig:Rigid_body_illustration}
\end{figure}

\newpage

\subsection{The Lotka-Volterra System}
Recall the behaviour of the PHI and RK-2 from section \ref{sec:LVnum} -- we now have the correct language to explain it, studying in particular the Casimirs of the system.

The Poisson structure of the generic Lotka-Volterra system coincides with the quadratic one of Equation \eqref{eq:quad-Poisson}, fully encoded in an $n\times n$ matrix $A$. The dynamics is governed by a linear Hamiltonian $H = \sum_{i=1}^n x_i$. 

\begin{proposition}
Let $u \in \text{Ker} A.$ 
\begin{align*}
  f \colon \mathbb{R}^n &\to \mathbb{R}\\
  x &\mapsto \prod_{1 \leq i \leq n}x_i^{u_i}
\end{align*}
is a local Casimir of the quadratic Poisson structure given by the matrix $A.$
\end{proposition}

In the numerical test, we considered $n = 3$ and,
$A = \begin{pmatrix}
0 & 1 & 1\\
-1 & 0 & 1\\
-1 & -1 & 0
\end{pmatrix}$
so that the local Casimir is $C \colon x \mapsto \frac{x_1 x_3}{x_2}.$ Generic symplectic leaves are hyperbolas, and intersecting them with the level surfaces of $H$ one obtains the geometry of the real trajectory (up to time parametrisation).
A Runge-Kutta method does not preserve $C$, while the constructed Poisson Hamiltonian integrator can preserve the Casimir value with machine precision and the Hamiltonian up to any given order in timestep. Clearly the singular behaviour of the trajectory can be observed only provided these conservation laws are respected. Figure \ref{fig:cas_RK2_vs_PHI1} enlightens the stability of a Poisson Hamiltonian integrator in the neighborhood of a singularity observed in Section \ref{sec:LVnum}: on top of preserving the Poisson structure, it stays on a symplectic leaf along iterations.

\begin{figure}[htp]
    \centering
    \includegraphics[width = \textwidth]{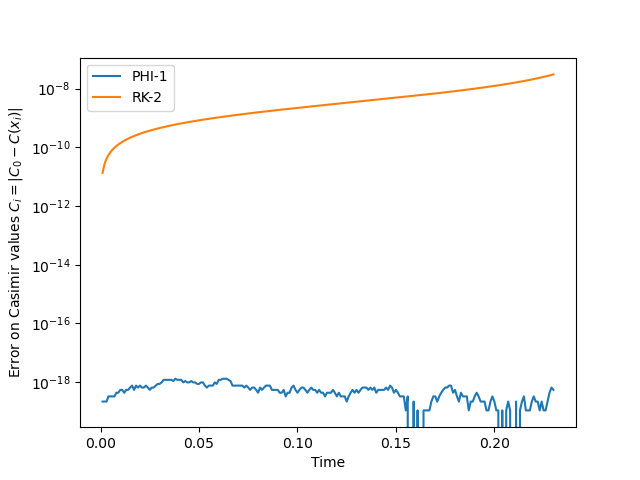}
    \caption{Comparison on Casimir values between PHI-1 and RK-2}
    \label{fig:cas_RK2_vs_PHI1}
\end{figure}

\newpage

\section*{Conclusion / perspectives}
In this paper, we have explained how the idea of the groupoid construction from \cite{oscar} can be implemented for design of Poisson integrators. Let us stress again that the term \emph{Poisson Hamiltonian integrators} we have introduced is important -- it explains the conceptual difference to straightforward constructions present in literature. 

We have seen that even for \emph{simple academic examples} in \emph{generic situations} constructed 
Poisson Hamiltonian integrators proved to be more  accurate than even higher order classical methods, especially on long run simulations. 
But a similar strategy can be implemented with no changes for more complicated systems of ordinary differential equations -- we are working on a symbolic package for automatic generation of the simulation source codes for that (\cite{OAV}). 
Moreover, similar methods can be designed even for Poisson Hamiltonian partial differential equations, which often appear in fluid dynamics and waves simulations. The key idea there is to use the locality of discretisation in space to spell-out the groupoid structure maps -- we intend to explore this direction in further works. 

\newpage
\textbf{Acknowledgments.} We are thankful to participants of the Geometry and Mechanics working group (La Rochelle, M2N team) and the Seminar on Geometry, Mechanics and Control (ICMAT -- IMAULL) for their valuable feedback.
The last section benefits from fruitful remarks of the CNRS Research Regroupement ``Differential Geometry and Mechanics''.
We appreciate enlightening discussion with Pol Vanhaecke, Dina Razafindralandy and Aziz Hamdouni at various stages of this work.
We are also thankful to Antoine Falaize for his help in implementation of the symbolic computations mentioned in remark \ref{ca}.

\end{document}